\documentclass[12pt]{amsart}
\usepackage[left=25truemm,right=25truemm,top=25truemm,bottom=25truemm]{geometry}
\usepackage{latexsym,amssymb,amsmath,amsthm,amscd,enumerate,graphicx,longtable}

\makeatletter
\@namedef{subjclassname@2010}{%
  \textup{2010} Mathematics Subject Classification}
\makeatother

\newtheorem{thm}{Theorem}[section]
\newtheorem{cor}[thm]{Corollary}
\newtheorem{lem}[thm]{Lemma}

\newtheorem{clm}[thm]{Claim}

\theoremstyle{definition}
\newtheorem{defin}[thm]{Definition}
\newtheorem{rem}[thm]{Remark}

\numberwithin{equation}{section}

\newcommand{\cD}{{\mathcal D}}

\newcommand{\cQ}{{\mathcal Q}}
\newcommand{\cR}{{\mathcal R}}

\newcommand{\cdq}{{\mathcal D\!\mathcal Q}}
\newcommand{\cdr}{{\mathcal D\!\mathcal R}}

\newcommand{\R}{{\mathbb R}}

\newcommand{\Z}{{\mathbb Z}}

\newcommand{\kK}{{\mathfrak K}}

\def\al{\alpha}

\def\gm{\gamma}

\def\dl{\delta}

\def\eps{\varepsilon}

\def\sg{\sigma}

\def\om{\omega}

\def\0{\emptyset}
\def\1{{\bf 1}}
\def\6{\partial}
\def\8{\infty}

\def\lt{\left}
\def\rt{\right}

\def\ol{\overline}

\def\lgl{\langle}
\def\rgl{\rangle}

\begin{document}

\title{The rectangular fractional integral operators}
\author[H.~Tanaka]{Hitoshi~Tanaka}
\address{
Research and Support Center on Higher Education for the Hearing and Visually Impaired, 
National University Corporation Tsukuba University of Technology, 
Kasuga 4-12-7, 
Tsukuba 05-8521, 
Japan
}
\email{htanaka@k.tsukuba-tech.ac.jp}

\thanks{
The author was supported by 
Grant-in-Aid for Scientific Research (C) (19K03510), 
the Japan Society for the Promotion of Science.
Part of this work was supported by the Research Institute for Mathematical Sciences,
an International Joint Usage/Research Center located in Kyoto University.
}

\keywords{
Carleson-type embedding;
Fefferman--Phong-type condition;
$M$-linear embedding theorem;
rectangular doubling weight;
rectangular fractional integral operator.
}
\subjclass[2010]{42B25,\,42B35.}
\date{}

\begin{abstract}
With rectangular doubling weight, a~generalized Hardy-Littlewood-Sobolev inequality for rectangular fractional integral operators is verified.
The result is a~nice application of $M$-linear embedding theorem for dyadic rectangles.
\end{abstract}

\maketitle

\section{Introduction}\label{sec1}
The purpose of this paper is to demonstrate a~nice application of $M$-linear embedding theorem for dyadic rectangles obtained in \cite{TY} and \cite{Ta4}, and to study a~generalized Hardy-Littlewood-Sobolev inequality for rectangular fractional integral operators.

For a~positive integer $N$, let $0<\al<N$.
For the rectangular doubling weight $\mu$\footnote{
For precise definition, see Subsection \ref{ssec3.1}
} on $\R^{N}$, define the rectangular the fractional integral operator $R_{\al}^{\mu}$ by
\[
R_{\al}^{\mu}f(x)
:=
\int_{\R^{N}}
\mu(R(x,y))^{\frac{\al}{N}-1}
f(y)\,{\rm d}\mu(y),
\quad x\in\R^{N},
\]
where $R(x,y)$ stands for the minimal rectangle, with respect to inculusion, which contains two deferent points, up to coordinates, $x$ and $y$ and has their sides parallel to the coordinate axes.
We have the following theorem the proof of which is our goal.

\begin{thm}\label{thm1.1}
For $1<p<q<\8$ with $\frac1q=\frac1p-\frac{\al}{N}$, the generalized Hardy-Littlewood-Sobolev inequality 
\[
\|R_{\al}^{\mu}f\|_{L^q(\mu)}
\lesssim
\|f\|_{L^p(\mu)}
\]
holds for all $f\in L^p(\mu)$.
\end{thm}

In the case $\mu \equiv 1$ and when we restrict rectangles to cubes, Theorem \ref{thm1.1} is just the Hardy-Littlewood-Sobolev inequality which is one of the most fundamental norm inequality of real variable harmonic analysis.
In the case $\mu$ is a~doubling weight and when we restrict rectangles to cubes, Theorem \ref{thm1.1} was studied in Stein's book \cite{St}.

Thanks to the fact that $\mu$ is rectangular doubling weight, the operator $R_{\al}^{\mu}$ almost commutes with a~multi-parameter family of dilations.
The weight theory of such product operators, commuting with a~multi-parameter family of dilations, 
is not developed so much up to now despite a~number of pioneering works in the 1980's due to Robert Fefferman and Elias Stein
(see \cite{Fe1,Fe2,FS}).
The area remains largely open for product fractional integrals
(see also \cite{CXY,KM,SZ,Wa}). 

The letter $C$ will be used for constants that may change from one occurrence to another.
Constants with subscripts, such as $C_1$, $C_2$, do not change in different occurrences.
By $A\approx B$ we mean that 
$c^{-1}B\le A\le cB$ 
with some positive finite constant $c$ independent of appropriate quantities. 
We write $X\lesssim Y$, $Y\gtrsim X$ 
if there is a independent constant $c$ such that $X \le cY$. 

\section{
$M$-linear embedding theorem for product dyadic cubes
}\label{sec2}
In what follows we recall $M$-linear embedding theorem for dyadic rectangles obtained in \cite{TY} and \cite{Ta4}. 

Let $(N_1,N_2,\ldots,N_n)$ be an~$n$-tuple of positive integers and let $N=\sum_{i=1}^nN_i$. We decompose the $N$ dimensional Euclidean space $\R^{N}$ by
\[
\R^{N}
=
\prod_{i=1}^n\R^{N_i}.
\]
By $x\in\R^{N}$ we will denote 
\[
x=(x_1,x_2,\ldots,x_n),
\quad x_i\in\R^{N_i}.
\]
By a~\emph{product cube} we will always mean a~rectangle $R\subset\R^{N}$ of the form
\[
R
=
\prod_{i=1}^nQ_i,
\quad Q_i\in\cQ(\R^{N_i}),
\]
where $\cQ(\R^d)$ stands for the set of all cubes in $\R^d$ with sides parallel to the coordinate axes. 
We will denote by $\cR(\R^{N})$ the family of all such product cubes. That is,
\[
\cR(\R^{N})
:=
\prod_{i=1}^n\cQ(\R^{N_i}).
\]
We denote by $\cdq(\R^d)$ the family of all dyadic cubes 
$Q=2^{-k}(m+[0,1)^d)$, 
$k\in\Z,\,m\in\Z^d$. 
We denote by $\cdr(\R^{N})$ 
the family of all product dyadic cubes on the product space $\R^{N}$, that is, 
\[
\cdr(\R^{N})
:=
\prod_{i=1}^n\cdq(\R^{N_i}).
\]

By a~\emph{weight} we will always mean a~nonnegative, locally integrable function  on the product space $\R^{N}$ which is positive on a~set of positive measure.
Given a~measurable set $E$ and a~weight $\om$, we will use the following notation:

\begin{itemize}
\item
Denote by $|E|$ the volume of $E$;
\item
The symbol $\om(E)$ denotes the quantity $\int_{E}\om(x)\,{\rm d}x$;
\item
The symbol $\1_{E}$ stands for the characteristic function of $E$;
\item
The measure of weight ${\rm d}\om$ is defined by ${\rm d}\om:=\om(x)\,{\rm d}x$.
\end{itemize}

\noindent
Let $1\le p<\8$ and $\om$ be a~weight. 
We define the weighted Lebesgue space $L^p(\om)$ to be a~Banach space equipped with the norm 
\[
\|f\|_{L^p(\om)}
:=
\lt(\int_{\R^{N}}|f|^p\,{\rm d}\om\rt)^{\frac1p}.
\]
Given $1<p<\8$, 
$p'=\frac{p}{p-1}$ denotes the conjugate exponent of $p$. 

\noindent\textbf{
The condition $({\rm D})$:
}\quad
We denote by $P_i$, $i=1,2,\ldots,n$, the projection onto the coordinate subspace $\R^{N_i}$. 
For a~product cube $R\in\cR(\R^{N})$,
an~integer $j=1,2,\ldots,n$ and 
a~cube $Q\in\cQ(\R^{N_j})$,
we define the product cube 
\[
\lgl R;\,Q,j\rgl
:=
\lt(\prod_{i=1}^{j-1}P_i(R)\rt)
\times Q\times
\lt(\prod_{i=j+1}^nP_i(R)\rt),
\]
which simply replaces $P_j(R)$ by $Q$.
For a~cube $Q\in\cQ(\R^d)$, let $\cD(Q)$ be the collection of all dyadic subcubes of $Q$, that is, all those cubes obtained by dividing $Q$ into $2^d$ congruent cubes of half its length, 
dividing each of those into $2^d$ congruent cubes, and so on. By convention, $Q$ itself belongs to $\cD(Q)$.

\begin{quotation}
For all positive number $\eps>0$, we assume that the weight $\sg$ satisfies \emph{the condition $({\rm D})$}: 
\begin{equation}\tag{{\rm D}}
\sum_{Q\in\cD(P_j(R))}
\sg(\lgl R;\,Q,j\rgl)^{1+\eps}
\le C_{\sg,\eps}
\sg(R)^{1+\eps}
\end{equation}
holds for all product cubes $R\in\cR(\R^{N})$ and all integers $j=1,2,\ldots,n$, where the constant $C_{\sg,\eps}$ depends only on the weight $\sg$ and the parameter $\eps$.
\end{quotation}

The following $M$-linear embedding theorem for product dyadic cubes is our crucial tool.
For the case of dyadic cubes, we refer to a~series of works 
\cite{HHL,Hy,LSU,NTV,Ta1,Ta2,Ta3,Tr}.

\begin{thm}
[{\rm\cite[Theorem 3.1]{Ta4}}]
\label{thm2.1}
Let $\kK:\,\cdr(\R^{N})\to[0,\8)$ be a~map. 
For $k=1,2,\ldots,M$, 
let $\sg_k$ be a~weight on $\R^{N}$ 
that satisfies the condition $({\rm D})$ and 
let $1<p_k<\8$ with 
$\sum_{k=1}^{M}\frac{1}{p_k}>1$.
The following statements are equivalent\text{:}

\begin{itemize}
\item[{\rm(a)}] 
The $M$-linear embedding inequality for product dyadic cubes 
\[
\sum_{R\in\cdr(\R^{N})}
\kK(R)\prod_{k=1}^{M}
\lt|\int_{R}f_k\,{\rm d}\sg_k\rt|
\le c_1
\prod_{k=1}^{M}
\|f_k\|_{L^{p_k}(\sg_k)}
\]
holds for all 
$f_k\in L^{p_k}(\sg_k)$, 
$k=1,2,\ldots,M$;
\item[{\rm(b)}] 
The Fefferman--Phong-type condition 
\[
\kK(R)
\prod_{k=1}^{M}
\sg_k(R)^{\frac{1}{p_k'}}
\le c_2
\]
holds for all product dyadic cubes $R\in\cdr(\R^{N})$.
\end{itemize}

\noindent
Moreover, 
the least possible constants $c_1$ and $c_2$ are equivalent.
\end{thm}

By a~simple duality argument for the case $M=2$ we have the following corollary.

\begin{cor}\label{cor2.2}
Let $\kK:\,\cdr(\R^{N})\to[0,\8)$ be a~map. 
Let $\om$ and $\sg$ be the weights on $\R^{N}$ that satisfy the condition $({\rm D})$ and let $1<p<q<\8$.
The following statements are equivalent\text{:}

\begin{itemize}
\item[{\rm(a)}] 
The weighted norm inequality for rectangular dyadic positive operator $T_{\kK}^{\sg}$
\[
\|T_{\kK}^{\sg}f\|_{L^q(\om)}
\le c_1
\|f\|_{L^p(\sg)}
\]
holds for all $f\in L^p(\sg)$.
Here, 
\[
T_{\kK}^{\sg}f(x)
:=
\sum_{R\in\cdr(\R^{N})}
\kK(R)\1_{R}(x)
\int_{R}f\,{\rm d}\sg,
\quad x\in\R^{N}.
\]
\item[{\rm(b)}] 
The Fefferman--Phong-type condition 
\[
\kK(R)
\om(R)^{\frac1q}
\sg(R)^{\frac{1}{p'}}
\le c_2
\]
holds for all product dyadic cubes $R\in\cdr(\R^{N})$.
\end{itemize}

\noindent
Moreover,
the least possible constants $c_1$ and $c_2$ are equivalent.
\end{cor}

We have the following corollary too.

\begin{cor}\label{cor2.3}
Let $0<\al<N$.
Let $\mu$ be a~weight on $\R^{N}$ that satisfies the condition $({\rm D})$ and let $1<p<q<\8$ with $\frac1q=\frac1p-\frac{\al}{N}$.
Then, the Hardy-Littlewood-Sobolev inequality for the rectangular dyadic positive operator $T_{\al}^{\mu}$
\[
\|T_{\al}^{\mu}f\|_{L^q(\mu)}
\lesssim
\|f\|_{L^p(\mu)}
\]
holds for all $f\in L^p(\mu)$\footnote{
We merely check that 
$\frac{\al}{N}-1+\frac1q+1-\frac1p
=
\frac{\al}{N}+\frac1q-\frac1p
=0$.
}.
Here, 
\[
T_{\al}^{\mu}f(x)
:=
\sum_{R\in\cdr(\R^{N})}
\mu(R)^{\frac{\al}{N}-1}\1_{R}(x)
\int_{R}f\,{\rm d}\mu,
\quad x\in\R^{N}.
\]
\end{cor}

\section{
Proof of Theorem \ref{thm1.1}
}\label{sec3}
In what follows we will prove Theorem \ref{thm1.1}.

\subsection{
Rectangular doubling weight
}\label{ssec3.1}
We need some definitions and remarks. 

\begin{defin}\label{def3.1}
For the cube $Q\in\cQ(\R^d)$, let 
\[
\cD^{(1)}(Q)
:=
\{Q'\in\cD(Q):\,2\ell(Q')=\ell(Q)\},
\]
where, by $\ell(Q)$, we denote the side-lengths of cube $Q$.

\begin{itemize}
\item
We say that a~weight $\sg$ on the product space $\R^{N}$ is a~\emph{doubling weight} if 
there is a~constant $\dl>0$ such that 
\[
\sg(R)\le\dl\sg(\lgl R;\,Q,j\rgl)
\]
holds for all product cubes $R\in\cR(\R^{N})$, all integers $j=1,2,\ldots,n$ and all dyadic cubes $Q\in\cD^{(1)}(P_j(R))$.
\item
We say that a~weight $\sg$ on the product space $\R^{N}$ is a~\emph{reverse doubling weight} if there is a~constant $\gm>1$ such that 
\[
\gm\sg(\lgl R;\,Q,j\rgl)\le\sg(R)
\]
holds for all product cubes $R\in\cR(\R^{N})$, all integers $j=1,2,\ldots,n$ and all dyadic cubes $Q\in\cD^{(1)}(P_j(R))$.
\end{itemize}
\end{defin}

\begin{rem}\label{rem3.2}
If the weight $\sg$ is doubling, then it is reverse doubling.
Indeed, for any $R\in\cR(\R^{N})$, any $j=1,2,\ldots,n$ and any $Q\in\cD^{(1)}(P_j(R))$, 
\[
\sg(R)
=
\sum_{Q'\in\cD^{(1)}(P_j(R))}
\sg(\lgl R;\,Q',j\rgl)
\ge
\lt(1+\frac{2^{N_j}-1}{\dl}\rt)
\sg(\lgl R;\,Q,j\rgl)
\ge
\lt(1+\frac{1}{\dl}\rt)
\sg(\lgl R;\,Q,j\rgl),
\]
where we have used 
\[
\sg(\lgl R;\,Q',j\rgl)
\ge
\frac{\sg(R)}{\dl}
\ge
\frac{\sg(Q)}{\dl}.
\]

Conversely, if the weight $\sg$ is reverse doubling and $\gm>2^{\max_iN_i}-1$, then it is doubling. 
Indeed, for any $R\in\cR(\R^{N})$, any $j=1,2,\ldots,n$ and any $Q\in\cD^{(1)}(P_j(R))$, 
\begin{align*}
\sg(R)
&=
\sum_{Q'\in\cD^{(1)}(P_j(R))}
\sg(\lgl R;\,Q',j\rgl)
\le
\frac{2^{N_j}-1}{\gm}\sg(R)
+
\sg(\lgl R;\,Q,j\rgl)
\\ &\le
\frac{2^{\max_iN_i}-1}{\gm}\sg(R)
+
\sg(\lgl R;\,Q,j\rgl).
\end{align*}
Hence,
\[
\sg(R)
\le
\frac{\gm}{\gm+1-2^{\max_iN_i}}
\sg(\lgl R;\,Q,j\rgl).
\]
\end{rem}

\begin{rem}\label{rem3.3}
If $\sg$ is a~reverse doubling weight on the product space $\R^{N}$, then it satisfies condition $({\rm D})$.
Indeed, for any product cube $R\in\cdr(\R^{N})$, any integer $j=1,2,\ldots,n$ and any positive number $\eps>0$,
\begin{align*}
\sum_{Q\in\cD(P_j(R))}
\sg(\lgl R;\,Q,j\rgl)^{1+\eps}
&=
\sum_{k=0}^{\8}
\sum_{\substack{
Q\in\cD(P_j(R)) 
\\ 
\ell(Q)=2^{-k}\ell(P_j(R))
}}
\sg(\lgl R;\,Q,j\rgl)^{1+\eps}
\\ &=
\sum_{k=0}^{\8}
\sum_{\substack{
Q\in\cD(P_j(R)) 
\\ 
\ell(Q)=2^{-k}\ell(P_j(R))
}}
\sg(\lgl R;\,Q,j\rgl)^{\eps}
\sg(\lgl R;\,Q,j\rgl)
\\ &\le
\sum_{k=0}^{\8}
\lt(\frac{1}{\gm^k}\rt)^{\eps}
\sg(R)^{\eps}
\sum_{\substack{
Q\in\cD(P_j(R)) 
\\ 
\ell(Q)=2^{-k}\ell(P_j(R))
}}
\sg(\lgl R;\,Q,j\rgl)
\\ &=
\sg(R)^{1+\eps}
\sum_{k=0}^{\8}
\lt(\frac{1}{\gm^k}\rt)^{\eps}
\\ &= C_{\gm,\eps}
\sg(R)^{1+\eps}.
\end{align*}
\end{rem}

\subsection{
The P\'{e}rez representation
}\label{ssec3.2}
For a~number $c>0$ and a~product cube $R\in\cR(\R^{N})$, we will use $cR$ to denote the product cube with the same center as $R$ but with $c$ times the side-lengths of $R$. 

We define a~C.~P\'{e}rez type representation of fractional integrals (see \cite{Pe}) by
\[
\ol{R}_{\al}^{\mu}f(x)
:=
\sum_{R\in\cdr(\R^{N})}
\mu(R)^{\frac{\al}{N}-1}\1_{R}(x)
\int_{3R}f\,{\rm d}\mu,
\quad x\in\R^{N}.
\]

We now verify the point-wise equivalence 
\begin{equation}\label{3.1}
R_{\al}^{\mu}f(x)
\approx
\ol{R}_{\al}^{\mu}f(x).
\end{equation}

We first observe that, for $u,v\in\R^d$ with $u \neq v$, the minimal dyadic cube $Q\in\cdq(\R^d)$ such that $Q\ni u$ and $3Q\ni v$ satisfies
\[
\frac{\ell(Q)}{2}<|u-v|<2\sqrt{d}\ell(Q).
\]
We will refer to such a~dyadic cube as $Q(u,v)$.

This observation, together with the fact that $\mu$ is a~doubling weight, enable us 
to deduce that
\[
\mu(R(x,y)) \approx \mu(R_0(x,y)),
\]
where, 
$x=(x_1,x_2,\ldots,x_n)\in\R^{N}$,
$y=(y_1,y_2,\ldots,y_n)\in\R^{N}$,
$x_i \neq y_i$, and 
\[
R_0(x,y)
:=
\prod_{i=1}^nQ(x_i,y_i).
\]
The fact that $\mu$ is a~reverse doubling weight and a~calculus of geometric series also enable us to deduce that
\[
\sum_{R\in\cdr(\R^{N})}
\mu(R)^{\frac{\al}{N}-1}
\1_{R}(x)\1_{3R}(y)
\approx
\mu(R_0(x,y))^{\frac{\al}{N}-1}
\approx
\mu(R(x,y))^{\frac{\al}{N}-1}.
\]
This equation and Fubini's theorem yield \eqref{3.1}.

\subsection{
The dyadic grid argument
}\label{ssec3.3}
For $\tau\in\{0,\pm\frac13\}^d$, the dyadic grid $\cdq^{\tau}(\R^d)$ is defined by 
\[
\cdq^{\tau}(\R^d)
:=
\{2^{-k}(m+\tau+[0,1)^d):\,
k\in\Z,m\in\Z^d\}.
\]

\begin{clm}\label{clm3.4}
We claim that, for any dyadic cube $Q\in\cdq(\R^d)$, there exist $\tau\in\{0,\pm\frac13\}^d$ and $\tau$-shifted dyadic cube $P\in\cdq^{\tau}(\R^d)$ such that 
$3Q\subset P$ and $\ell(P)=8\ell(Q)$.
\end{clm}

\begin{proof}
We need only verify the one-dimensional case $d=1$. 
(The claim for $d>1$ holds after $d$ steps.)
We may assume further $k=0$.
Let $Q=[m,m+1)$, $m\in\Z$. 
Then $3Q=[m-1,m+2)$.
We cover $3Q$ by disjoint dyadic intervals of $\cD(\R)$ 
with the same length $8$.
If $3Q$ is covered by such an~interval $P$, 
then we choose $\tau=0$ and have 
$P\in\cD^{\tau}(\R)$. 
We assume that $3Q$ is covered by such two intervals as 
$P_1\ni(m-1)$ 
and 
$P_2\ni(m+2)$.
If $|Q\cap P_1|\ge 1.5$, then
we choose $\tau=\frac13$ and let 
$P=\frac83+P_1$.
If $|Q\cap P_2|>1.5$, then
we choose $\tau=-\frac13$ and let 
$P=-\frac83+P_2$. 
This proves the claim.
\end{proof}

\subsection{
Proof of Theorem \ref{thm1.1}
}\label{ssec3.4}
Using Claim \ref{clm3.4}, 
we see that for any product dyadic cube $R\in\cdr(\R^{N})$ there exist 
$\tau\in\{0,\pm\frac13\}^{N}$ 
and 
$R'\in\cdr^{\tau}(\R^{N})$, 
which is defined by 
\[
\cdr^{\tau}(\R^{N})
:=
\prod_{i=1}^n
\cdq^{P_i(\tau)}(\R^{N_i}),
\]
such that $3R\subset R'$ and 
$\ell(P_i(R'))=8\ell(P_i(R))$,
$i=1,2,\ldots,n$. 
Moreover, 
we notice that this correspondence
$R \mapsto R'$
becomes \emph{almost bijective}.

By letting 
\[
T_{\al}^{\mu,\tau}f(x)
:=
\sum_{R\in\cdr^{\tau}(\R^{N})}
\mu(R)^{\frac{\al}{N}-1}\1_{R}(x)
\int_{R}f\,{\rm d}\mu,
\quad x\in\R^{N},
\]
since $\mu$ is doubling,
\begin{align*}
\ol{R}_{\al}^{\mu}f(x)
&=
\sum_{R\in\cdr(\R^{N})}
\mu(R)^{\frac{\al}{N}-1}\1_{R}(x)
\int_{3R}f\,{\rm d}\mu
\\ &\lesssim
\sum_{\tau\in\{0,\pm\frac13\}^{N}}
T_{\al}^{\mu,\tau}f(x).
\end{align*}
Applying Corollary \ref{cor2.3} to the right-hand side of this inequality, we obtain Theorem \ref{thm1.1}.

\section*{Appendix}
As an~appendix, 
we state the Carleson-type embedding theorem for product dyadi
c cubes. 
The importance of this theorem is that 
the Fefferman--Phong-type condition 
simply links 
the Carleson-type embedding theorem 
to 
the $M$-linear embedding theorem.

\begin{lem}
[{\rm\cite[Lemma 2.2]{Ta4}}]
Given a~weight $\sg$ in the product space $\R^{N}$ 
and $1<p<q<\8$, 
the following statements are equivalent\text{:}

\begin{itemize}
\item[{\rm(a)}] 
The Carleson-type embedding inequality for product dyadic cubes
\[
\sum_{R\in\cdr(\R^{N})}
\sg(R)^{\frac{q}{p}}
\lt(\frac{1}{\sg(R)}\int_{R}f\,{\rm d}\sg\rt)^q
\le c_1
\lt(\int_{\R^{N}}f^p\,{\rm d}\sg\rt)^{\frac{q}{p}}
\]
holds for all nonnegative functions 
$f\in L^p(\sg)$;
\item[{\rm(b)}] 
The testing condition 
\[
\sum_{Q\in\cD(P_j(R))}
\sg(\lgl R;\,Q,j\rgl)^{\frac{q}{p}}
\le c_2
\sg(R)^{\frac{q}{p}}
\]
holds for all product dyadic cubes $R\in\cdr(\R^{N})$ 
and all integers $j=1,2,\ldots,n$.
\end{itemize}

\noindent
Moreover, 
the least possible constants $c_1$ and $c_2$ 
enjoy $c_1\le C c_2^n$ 
and $c_2\le c_1$.
\end{lem}

\end{document}